\documentclass[11pt]{amsart}

\usepackage{verbatim}
\usepackage{amsmath,amsthm,amssymb,amscd}
\usepackage{mathrsfs}
\usepackage[all]{xy}
\usepackage{xspace}
\usepackage{ulem}
\usepackage{enumerate}
\usepackage{todonotes}
\usepackage{xcolor}
\usepackage{hyperref}

\newcommand{\R}{\ensuremath{\mathbb{R}}\xspace}

\newcommand{\pr}[2]{\ensuremath{\langle {#1},{#2}\rangle}}
\newcommand{\norm}[1]{\ensuremath{\|#1\|}}

\newcommand{\ma}{\ensuremath{\mathcal{A}}\xspace}
\newcommand{\mb}{\ensuremath{\mathcal{B}}\xspace}

\newcommand{\mn}{\ensuremath{\mathcal{N}}\xspace}

\newcommand{\nuc}[1]{\ensuremath{\Bbbk ({#1})}\xspace}
\newcommand{\nucc}[1]{\ensuremath{\Bbbk_c({#1})}\xspace}

\newcommand{\bts}[3]{\ensuremath{{#1}\bigotimes_{#2}{#3}}\xspace}
\newcommand{\bats}[2]{\ensuremath{{#1}\bigodot{#2}}\xspace}
\newcommand{\tmin}[2]{\ensuremath{{#1}\bigotimes_{\min}{#2}}\xspace}
\newcommand{\tmax}[2]{\ensuremath{{#1}\bigotimes_{\max}{#2}}\xspace}

\newcommand{\al}{\ensuremath{\alpha}\xspace}

\newcommand{\la}{\ensuremath{\lambda}\xspace}

\newcommand{\La}{\ensuremath{\Lambda}\xspace}

\newcommand{\ov}[1]{\overline{#1}}

\newcommand{\cs}{$C^*$-algebra\xspace}

\newcommand{\fbs}{Fell bundles\xspace}

\newcommand{\fela}{\ensuremath{\ma=(A_t)_{t\in G}}\xspace}

\newcommand{\felbh}{\ensuremath{\mb=(B_s)_{s\in H}}\xspace}

\renewcommand{\k}[1]{\mathcal{K}({#1})}

\newcommand{\bc}{\begin{center}}
\newcommand{\ec}{\end{center}}
\newcommand{\be}{\begin{enumerate}}
\newcommand{\ee}{\end{enumerate}}
\newcommand{\bi}{\begin{itemize}}
\newcommand{\ei}{\end{itemize}}
\newcommand{\bd}{\begin{description}}
\newcommand{\ed}{\end{description}}
\newcommand{\beq}{\begin{equation}}
\newcommand{\eeq}{\end{equation}}
\newcommand{\beqa}{\begin{eqnarray}}
\newcommand{\eeqa}{\end{eqnarray}}
\newcommand{\bfr}{\begin{flushright}}
\newcommand{\efr}{\end{flushright}}
\newcommand{\bfl}{\begin{flushleft}}
\newcommand{\efl}{\end{flushleft}}
\newcommand{\bp}{\begin{picture}}
\newcommand{\ep}{\end{picture}}

\DeclareMathOperator{\gen}{span}

\title{Kernels of Fell bundles tensor products}

\address{Centro de Matem\'atica\\ 
         Facultad de Ciencias\\
         Universidad de la Rep\'ublica\\
         Igu\'a 4225\\
         CP 11400\\
         Montevideo--URUGUAY}
\email{fabadie@cmat.edu.uy}

\theoremstyle{plain}
\newtheorem{thm}{Theorem}[section]
\newtheorem{prop}[thm]{Proposition}
\newtheorem{lem}[thm]{Lemma}
\newtheorem{cor}[thm]{Corollary}          

\theoremstyle{definition}
\newtheorem{df}[thm]{Definition}

\theoremstyle{remark}
\newtheorem{rk}[thm]{Remark}

\date{}

\begin{document}

\author[F. Abadie]{Fernando Abadie}
\begin{abstract}
  We prove that if $\ma=(A_t)_{t\in G}$ and $\mb=(B_s)_{s\in H}$ are Fell bundles over the locally compact groups $G$ and $H$ respectively, then the minimal (maximal) tensor product of the C*-algebra of kernels of $\ma$ with the C*-algebra of kernels of $\mb$ agrees with the C*-algebra of kernels of the minimal (respectively: maximal) tensor product of $\ma$ and $\mb$.  

\end{abstract}

\maketitle

\section{Introduction}\label{sec:intro}
\par In \cite{tens} we developed a theory of tensor products of Fell bundles, and we studied the relationship betweeen the cross-sectional C*-algebras of these tensor products with the tensor products of the corresponding bundles. In this article we will do something similar but in relation to the algebras of kernels of the bundles instead of cross-sectional algebras. The C*-algebra of kernels of a Fell bundle $\ma$ can be described (in addition to the descrption we will give below) as $\nuc{\ma}=C^*_r(\ma)\rtimes_{\delta}\hat{G}$ where $\delta$ is the dual coaction of $G$ on $C^*_r(\ma)$, and it helps to express a form of Takai duality for crossed products by partial actions. Besides, there is a canonical action $\beta$ of $G$ on $\nuc{\ma}$ (which corresponds to the dual action of $G$ on $C^*_r(\ma)\rtimes_{\delta}\hat{G}$ in the picture above) which plays an important role in the problem of globalization of partial actions (we refer the reader to \cite{env}).    
For this reason one can also interpret the results of this article with an eye on the tensor products of partial actions.
\medskip
\par We want to show that $\nuc{\tmin{\ma}{\mb}}=\tmin{\nuc{\ma}}{\nuc{\mb}}$, and similarly for the maximal C*-tensor product. The strategy will be the following. As shown in \cite{env}, one can see the C*-algebra $I(\ma)$ of compact operators on the Hilbert module $L^2(\ma)$ as an ideal of $\nuc{\ma}$, and we know the linear $\beta$-orbit of $I(\ma)$ is dense in $\nuc{\ma}$. This means that $\beta$ is the enveloping action of its restriction to $I(\ma)$ and, something we will take advantage of later, it allows to compute the norm of $k\in\k{\ma}$ as the supremum of the norms of $k$ as a multiplier of the ideals $\beta_t(\ma)$ (see Proposition~\ref{prop:norm} below). In this way one essentially reduces the problem to proving that $I(\tmin{\ma}{\mb})=\tmin{I(\ma)}{I(\mb)}$ and $I(\tmax{\ma}{\mb})=\tmax{I(\ma)}{I(\mb)}$. For this we will serve ourselves from the results in \cite{trings} and \cite{tens}.  
\bigskip
\par Before we start we will recall some facts about tensor products and about Fell bundles. 
\par Let us first recall the notion of exterior tensor product of Hilbert modules. So suppose that $E$ is a Hilbert module over the C*-algebra $A$, and $F$ is a Hilbert module over the C*-algebra $B$. Let $\bats{E}{F}$ be the algebraic tensor product of $E$ and $F$. Using the universal property of tensor products, it can be shown that there exists a unique bilinear map $(\bats{E}{F})\times (\bats{E}{F})\to\bats{A}{B}$ such that
\begin{equation}\label{eqn:bil}
  \pr{x\odot y}{x'\odot y'}=\pr{x}{x'}\odot\pr{y}{y'},\ \forall x,x'\in E,\  y,y'\in F.
\end{equation}

Besides, it can be shown that this map is actually a $\bats{A}{B}$-valued inner product (see \cite[page 34 and Chapter~6]{l}). Then, considering on $\bats{A}{B}$ the minimal C*-norm $\norm{\ }_{\textrm{min}}$ (which is the spatial norm), one can perform the so called double completion process (\cite[Chapter~1]{l}) and obtain a Hilbert $\tmin{A}{B}$-module $\tmin{E}{F}$, called the \textit{exterior tensor product} of $E$ and $F$. However, instead of $\norm{\ }_{\textrm{min}}$, one can use another C*-norm $\alpha$ on $\bats{A}{B}$ to perform the double completion process, for instance the maximal C*-norm $\norm{\ }_{\textrm{max}}$. This induces a norm $\tilde{\alpha}$ on $\bats{E}{F}$, and the corresponding completion $\bts{E}{\tilde{\alpha}}{F}$ turns out to be a Hilbert $\bts{A}{\alpha}{B}$-module. Of course, the relation between $\alpha$ and $\tilde{\alpha}$ is, as usual
\begin{gather}\label{eqn:alphas}
  \tilde{\alpha}(z)^2=\alpha(\pr{z}{z}),  z\in \bts{E}{\tilde{\alpha}}{F}\quad\textrm{ and }\quad\alpha(c)=\sup\{\tilde{\alpha}(zc): \tilde{\alpha}(z)\leq 1.\}
  \end{gather}
  \par In \cite{trings} it was shown that (provided $E$ and $F$ are full modules) the map $\alpha\mapsto\tilde{\alpha}$ is an isomorphism between the poset $\mn(\bats{A}{B})$ and the poset $\mn(\bats{E}{F})$ of C*-norms on $\bats{A}{B}$ and $\bats{E}{F}$ respectively (\cite[Theorem~5.12]{trings}). In particular we have $\widetilde{\norm{\ }}_{\textrm{min}}=\norm{\ }_{\textrm{min}}$ and $\widetilde{\norm{\ }}_{\textrm{max}}=\norm{\ }_{\textrm{max}}$. The work in \cite{trings} uses the language of C*-trings rather than that of Hilbert modules. For the topic of C*-trings we refer the reader to \cite{z}, where Zettl introduced and studied them.

\par Going now to define the tensor product of two Fell bundles $\ma=(A_r)_{r\in G}$ and $\mb=(B_s)_{s\in H}$, we begin by considering the algebraic tensor products $\bats{A_r}{B_s}$ and, given a C*-norm $\alpha$ on $\bats{A_e}{B_e}$, we want to extend it to all of the bundle $\bats{\ma}{\mb}=\{\bats{A_r}{B_s}\}$, so that we obtain a Fell bundle after completing. When this is possible, we call such a C*-norm \textit{extendable}. The condition for a C*-norm to be extendable is simple, and it is essentially the fact that when we perform the construction of $\bts{A_r}{\tilde{\alpha}}{B_s}$, the result is the same if we do it as a right or as left Hilbert $\bts{A_e}{\alpha}{B_e}$-module. Fortunately, the minimal and maximal norms on $\bts{A_e}{}{B_e}$ are both extendable. As shown in \cite{tens}, the map $\gamma\mapsto\gamma|_{\bats{A_e}{B_e}}$ is an isomorphism of posets, between the set $\mn(\bats{\ma}{\mb})$ of C*-norms on $\bats{\ma}{\mb}$ and the set $\mn(\bats{A_e}{B_e})$ of extendable C*-norms on $\bats{A_e}{B_e}.$ 

\medskip
\par We end this section with the following result about Hilbert modules, which is certainly well-known. A proof, in terms of C*-trings, can be found in \cite[Proposition~4.1]{env}.

\begin{prop}\label{prop:zhm}
Let $E$ be a full right Hilbert $A$-module and $F$ a right Hilbert $B$-module. If $\pi:E\to F$ is a linear map such that $\pi(x\pr{y}{z})=\pi(x)\pr{\pi(y)}{\pi(z)}$ $\forall x,y,z\in E$, then there exists a unique homomorphism $\pi^r:A\to B$ of C*-algebras such that $\pi^r(\pr{y}{z})=\pr{\pi(y)}{\pi(z)}$ $\forall y,z\in E$, and $\pi(xa)=\pi(x)\pi^r(a)$, $\forall x\in E$, $a\in A$. Moreover $\pi^r$ is injective if and only if so is $\pi$. An analogous result holds for left Hilbert modules instead of right ones. 
\end{prop}

\begin{rk}\label{rk:zhm}
A particular case that we will encounter later is the following. Suppose that $E$ is a right Hilbert $A$-module, and let $\k{E}$ be the C*-algebra of compact operators. Then, if $\theta_{x,y}:E\to E$ represents the compact operator on $E$, so that $\theta_{x,y}(z):=z\pr{x}{y}$, it is well known that $(E,\theta_{\cdot,\cdot })$ is a full left Hilbert $\k{E}$-module. Then, if $F$ is another Hilbert $A$-module and $\pi:E\to F$ is a linear map such that $\pi(\pr{x}{y})=\pr{\pi(x)}{\pi(y)}$, $\forall x,y\in E$, by the left version of Proposition~\ref{prop:zhm} we get a homomorphism of C*-algebras $\pi^l:\k{E}\to \k{F}$ such that $\pi^l(\theta_{x,y})=\theta_{\pi(x),\pi(y)}$, $\forall x,y\in E$. In particular $\pi^l$ is an isomorphism whenever so is $\pi$ (we note in passing that $\pi^r=id$).   
\end{rk}

\section{Kernels of a Fell bundle}\label{sec:kers}
\par Given a Fell bundle $\ma=(A_r)_{r\in G}$ over the locally compact group $G$, by a continuous kernel of compact support of $\ma$ we mean a continuous map $k:G\times G\to \ma$ of compact support such that $k(r,s)\in A_{rs^{-1}}$, $\forall r,s\in G$. The vector space $\nucc{\ma}$ of kernels of compact support of $\ma$ is actually an $A_e$-module with $ka(r,s):=k(r,s)a$, $\forall r,s\in G$, $a\in A_e$, $k\in\nucc{\ma}$. We also have operations making $\nucc{\ma}$ a $*$-algebra, namely: a product given by $k_1*k_2(r,s):=\int_Gk_1(r,t)k_2(t,s)dt$, and an involution given by $k^*(r,s):=k(s,r)^*$. In addition, $\nucc{\ma}$ is a normed $*$-algebra with $\norm{k}_{2}:=\big(\int_{G\times G}\norm{k(r,s)}^2d(r,s)\big)^{1/2}$, where the integration is with respect to the Haar measure of $G\times G$. The completion of $\nucc{\ma}$, which is a Banach $*$-algebra denoted by $\mathcal{HS}(\ma)$, is called the algebra of the kernels of Hilbert-Schmidt of $\ma$, and its universal enveloping C*-algebra $\nuc{\ma}:=C^*(\mathcal{HS}(\ma))$, is the C*-algebra of kernels of $\ma$.

\subsection{Canonical action on the kernels.} It is clear that if $k\in\nucc{A}$ and $t\in G$, then $\beta_t(k):G\times G\to\ma$, such that $\beta_t(r,s):=\Delta(t)k(rt,st)$ $\forall r,s\in G$, is also an element of $\nucc{\ma}$ (here $\Delta$ is the modular function of $G$). It follows that $k\mapsto\beta_t(k)$ is an action of $G$ on $\nucc{\ma}$. Moreover, we have $\norm{\beta_t(k)}_2=\norm{k}_2$, so the action extends to a an isometric action of $G$ on $\mathcal{HS}(\ma)$, and hence to an action of $G$ on $\nuc{\ma}$. If we need to remark that the underlying Fell bundle is $\ma$, we write $\beta^\ma$ instead of $\beta$. 
\subsection{Hilbert modules of a Fell bundle.} Given $t\in G$, let $C_c^t(\ma)$ be the vector space of those continuous functions $\xi:G\to\ma$ with compact support such that $\xi(r)\in A_{rt}, \forall r\in G$ (that is: $C_c^t(\ma)=C_c(\ma^t)$, where $\ma^t$ is the retraction of $\ma$ by the left translation $r\mapsto rt$ on $G$).  
\par We have natural actions $C_c^t(\ma)\times A_e\to C_c^t(\ma)$ and $\nucc{\ma}\times C_c^t(\ma)\to C_c^t(\ma)$, given by $(\xi a)(r):=\xi(r)a$ and $k\xi(r):=\int_Gk(r,s)\xi(s)ds$, $\forall \xi\in C_c^t(\ma)$, $r\in G$, $k\in\nucc{\ma}$. In this way $C_c^t(\ma)$ is a $(\nucc{\ma}-A_e)$-bimodule. We also have an inner product $\pr{\,}{}^t_r:C_c^t(\ma)\times C_c^t(\ma)\to A_e$, given by $\pr{\xi}{\eta}_r^t:=\int_G\xi(r)^*\eta(r)dr$.
\par We denote by $L^2_t(\ma)$ the Hilbert $A_e$-module obtained by completing $C_c^t(\ma)$ with respect to $\pr{\,}{}_r^t$. Note that $L^2_e(\ma)$ is the usual Hilbert module $L^2(\ma)$, where the left regular representation of $\ma$ takes place. 
When $t=e$ we just put $\pr{\xi}{\eta}_r$ and $L^2(\ma)$ instead of $\pr{\xi}{\eta}^e_r$ and $L^2_e(\ma)$ respectively.

\begin{df}
We denote by $I_t(\ma):=\k{L^2_t(\ma)}$ ( with the usual exception for $t=e$, when we will put in general $I(\ma)$ instead of $I_e(\ma)$).  
\end{df}

\medskip 
\begin{prop}\label{prop:isomhb}
The map $\rho_t:C_c(\ma)\to C_c^t(\ma)$, such that $\rho_t(\xi)(r):=\Delta(t)^{1/2}(\xi)(rt)$ $\forall \xi\in C_c(\ma)$, $r\in G$, extends to a unitary operator of Hilbert $A_e$-modules $L^2(\ma)\to L^2_t(\ma)$, and it is also an isomorphism of C*-trings. 
\end{prop}
\begin{proof}
  It is clear that $\rho_t$ is an $A_e$-linear map. Let $\xi,\eta\in C_c(\ma)$. Then
  \begin{gather*}
    \pr{\rho_t(\xi)}{\rho_t(\eta)}_r^t=\int_G\rho_t(\xi)(r)^*\rho_t(\eta)(r)dr
    =\int_G\Delta(t)^{1/2}\xi(rt)^*\Delta(t)^{1/2}\eta(rt)dr\\
    =\Delta(t)\int_G\xi(rt)^*\eta(rt)dr
    =\Delta(t)\int_G\Delta(t)^{-1}\xi(s)^*\eta(s)dr
    =\pr{\xi}{\eta}_r.
  \end{gather*}
  Then $\rho_t$ is isometric and can be extended by continuity to an isometry on $L^2(\ma)$. Thus $\rho_t$ has closed range, which clearly contains $C_c^t(\ma)$, and therefore $\rho_t$ is surjective. Hence $\rho_t$ is unitary \cite[Theorem~3.5]{l}. The last statement is a direct consequence of the first one.   
\end{proof}

\par As a consequence of Corollary~\ref{cor:isomhb} ans Proposition~\ref{prop:zhm} we get:

\begin{cor}\label{cor:isomhb}
The map $\rho_t^l:I(\ma)\to I_t(\ma)$ 
is an isomorphism of C*-algebras. It is determined by the fact that $\rho_t^l(\pr{\xi}{\eta}_l)=\pr{\rho_t(\xi)}{\rho_t(\eta)}_l$, $\forall \xi,\eta\in L^2(\ma)$. 
\end{cor}

\subsection{Producing kernels.}
For each $t\in G$, it is posible to obtain kernels of $\ma$ using elements of $C_c^t(\ma)$. In fact, if $\xi,\eta\in C_c^t(\ma)$, it is clear that $\pr{\xi}{\eta}^t_l:G\times G\to \ma$ such that $\pr{\xi}{\eta}^t_l(r,s)=\xi(r)\eta(s)^*$ is an element of $\nucc{\ma}$. We define $I_c^t(\ma):=\textrm{span}\{\pr{\xi}{\eta}^t_l:\xi,\eta\in C_c^t(\ma)\}\subseteq\nucc{\ma}$. When $t=e$ we just put $\pr{\xi}{\eta}_l$ and $I_c(\ma)$ instead of $\pr{\xi}{\eta}^e_l$ and $I_c^e(\ma)$ respectively.
\par Note that if $\xi,\eta\in C_c(\ma)$, we have $\beta_t(\pr{\xi}{\eta}_l)(r,s)=\Delta(t)\xi(rt)\eta(st)^*=\Delta(t)^{1/2}\xi(rt)\Delta(t)^{1/2}\eta(st)^*=\pr{\rho_t\xi}{\rho_t\eta}_l^{t}(r,s)$, so $\beta_t(\pr{\xi}{\eta}_l)=\pr{\rho_t\xi}{\rho_t\eta}_l^{t}$. It follows that $\beta_t(I_c(\ma))=I_c^{t}(\ma)$, $\forall t\in G$.  Observe also that each $I_c^t(\ma)$ is an ideal in $\nucc{\ma}$, $\forall t\in G$, because $I_c^t(\ma)$ is selfadjoint and $k*\pr{\xi}{\eta}_l^t=\pr{k\xi}{\eta}_l^t$, as it is easy to check. In a sense, the family $I_c^t(\ma)$ contains all the necessary information about $\nucc{\ma}$, since $\textrm{span}\{\beta_tI_c(\ma):t\in G\}=\textrm{span}\{I_c^t(\ma):t\in G\}$ is dense in $\nucc{\ma}$ in the inductive limit topology (\cite[Proposition~5.3]{env}).   
\begin{rk}
Lemma~5.3 of \cite{env} shows that $\pr{\xi}{\xi}_l$ is a positive element of $\nuc{\ma}$, $\forall \xi\in C_c(\ma)$. This implies that $\pr{\ }{}_l$ is an inner product on $C_c(\ma)$, and allows to complete $C_c(\ma)$ using $\pr{\,}{}_l$, as done in \cite{env}. Let us call $E$ to such a completion, and $\gamma$ to its norm. Then $\gamma$ induces a C*-norm $\gamma^r$ on the dense ideal $\textrm{span}\pr{C_c(\ma)}{C_c(\ma)}_r$ of $A_e$, which has a unique extension to a C*-norm on $A_e$, which in turn is unique, say $\alpha$ (see for instance \cite[Proposition~3.7, Corollary~3.8]{trings}). Since the relation between $\gamma$ and $\gamma^r$ is given by $\gamma(\xi)^2=\gamma^r(\pr{\xi}{\xi}_r)$ and, in the other hand, the norm defined on $C_c(\ma)$ by $\alpha$ is given by $\tilde{\alpha}(\xi)^2=\alpha(\pr{\xi}{\xi}_r)$, we see that $\gamma$ and $\tilde{\alpha}$ agree on $C_c(\ma)$, and therefore we conclude that $E=L^2(\ma)$. In turn, this implies that $I(\ma)$ is contained in $\nuc{\ma}$.     
\end{rk}

\begin{prop}\label{prop:extnuc}
  Let $\xi,\eta\in C_c^t(\ma)$. Then the map $\zeta\mapsto \pr{\xi}{\eta}_l^t\zeta$ on $C_c^t(\ma)$ extends by continuity to an adjointable (and compact) operator on $L^2_t(\ma)$. 
\end{prop}
\begin{proof}
Recall that any couple of elements $\eta,\zeta\in L^2_t(\ma)$ define a compact operator $\theta_{\eta,\zeta}$ on $L^2_t(\ma)$, via the formula $\theta_{\eta,\zeta}(\xi)=\xi\pr{\eta}{\zeta}_r^t$. So it is enough to see that for $\zeta\in C_c^t(\ma)$, we have $\pr{\xi}{\eta}^t_l\zeta=\theta_{\xi,\eta}(\zeta)$, that is $\pr{\xi}{\eta}^t_l\zeta=\xi\pr{\eta}{\zeta}_r$, $\forall \xi,\eta$ and $\zeta\in C_c(\ma)$. Now,
\begin{gather*}
  \pr{\xi}{\eta}^t_l\zeta(r)=\int_G\pr{\xi}{\eta}^t_l(r,s)\zeta(s)ds
  =\int_G\xi(r)\eta(s)^*\zeta(s)ds
  =\xi(r)\int_G\eta(s)^*\zeta(s)ds\\
  =\xi(r)\pr{\eta}{\zeta}_r^t
    =(\xi\pr{\eta}{\zeta}_r^t)(r), 
\end{gather*}
as we wanted to verify. 

\end{proof}

\par According to Proposition~\ref{prop:extnuc}, $I_t(\ma)$ is a completion of $I_c^t(\ma)=\beta_t(I_c(\ma))$. More precisely, we have  
\begin{prop}\label{prop:kt}
  $I_t(\ma)=\beta_t(I(\ma))$. 
\end{prop}
\begin{proof}
  Proposition~\ref{prop:isomhb} provides a unitary operator $\rho_t:L^2(\ma)\to L^2_t(\ma)$, given on $\xi\in C_c(\ma)$ by $\rho_t\xi(r):=\Delta(t)^{1/2}\xi(rt)$. Since $\rho_t$ is an isomorphism of C*-trings, it induces an isomorphism $\rho_t^l:I(\ma)\to I_t(\ma)$, unique such that $\rho_t^l(\pr{\xi}{\eta}_r)=\pr{\rho_t\xi}{\rho_t\eta}^t_r)$ (Proposition~\ref{prop:zhm}). On the other hand, it was shown above that $\pr{\rho_t\xi}{\rho_t\eta}_l^t=\beta_t(\pr{\xi}{\eta}_l)$, $\forall \xi,\eta\in C_c(\ma)$, which entails $\beta_t=\rho_t^l|_{I_c(\ma)}$, and therefore
  \begin{gather*}
    I_c^t(\ma)=\textrm{span}\pr{C_c^t(\ma)}{C_c^t(\ma)}_l^t
    =\textrm{span}\beta_t(\pr{C_c(\ma)}{C_c(\ma)}_l) 
    =\beta_t(I_c(\ma)).
  \end{gather*}
    Then $I_t(\ma)
    =\overline{I_c^t(\ma)}=\overline{\beta_t(I_c(\ma))}
    =\beta_t(\overline{I_c(\ma)})=\beta_t(I(\ma)).$

\end{proof}

\par For the convenience of the reader we recall Lemma~2.1 of \cite{env}:

\begin{lem}\label{lem:unienv}
Suppose that $\{J_{\la}\}_{\la\in\La}$ is a family of ideals in a \cs $A$, and let 
$\norm{\cdot}_{\La}:A\to\R$ be such that 
$\norm{a}_{\La}=\sup_{\la\in\La}\{\norm{ax}:\, x\in J_\la,\norm{x}\leq 1\}$. 
Then $\norm{\cdot}_{\La}$ is a $C^*$-seminorm on $A$, such that 
$\norm{\cdot}_{\La}\leq\norm{\cdot}$, and $\norm{\cdot}_{\La}$ is a norm iff
$\ov{\gen}\{ x\in J_{\la}:\la\in\La\}$ is an essential ideal of $A$. In this 
case, $\norm{\cdot}_{\La}=\norm{\cdot}$.
\end{lem}
As easy consequences we obtain: 
\begin{prop}\label{prop:norm}
  Let $k\in\nuc{\ma}$. Then
  \begin{equation}\label{eqn:norm}
    \norm{k}=\sup_{t\in G}\{\norm{k*h}:h\in I_t(\ma),\,t\in G\textrm{ and }\norm{h}\leq 1\}.\end{equation}
\end{prop}
\begin{prop}\label{prop:alphanorm}
  Let $\alpha$ be either the minimal or the maximal C*-norm on $\bats{\nuc{\ma}}{\nuc{\mb}}$, and consider the action $\beta:(G\times H)\times \bts{\nuc{A}}{\alpha}{\nuc{\mb}}\to \bts{\nuc{A}}{\alpha}{\nuc{\mb}}$ be the corresponding tensor product action $\beta^\ma\otimes\beta^\mb$, so that $\beta_{(r,s)}(k\otimes k')=\beta_r^\ma(k)\otimes\beta_s^\mb(k')$. Then  $\bts{I(\ma)}{\alpha}{I(\mb)}$ is an ideal of $\bts{\nuc{A}}{\alpha}{\nuc{\mb}}$, whose linear $\beta$-orbit is dense in $\bts{\nuc{A}}{\alpha}{\nuc{\mb}}$. Besides,
  \begin{equation}\label{eqn:alphanorm}\norm{k}=\sup_{t\in G}\{\norm{k*h}:h\in \bts{I_r(\ma)}{\alpha}{I_s(\mb)},\, r\in G,\,s\in H\textrm{ and }\norm{h}\leq 1\}.
  \end{equation}
\end{prop}
\begin{proof}
  It is clear that $\bts{I(\ma)}{\alpha}{I(\mb)}$ is an ideal of $\bts{\nuc{A}}{\alpha}{\nuc{\mb}}$, and the fact that the linear span of the orbit of $\bts{I(\ma)}{\alpha}{I(\mb)}$ is dense in $\bts{\nuc{\ma}}{\alpha}{\nuc{\mb}}$ follows at once from the density of the linear orbits of $I(\ma)$ and $I(\mb)$ in $\nuc{\ma}$ and $\nuc{\mb}$ respectively. Then \eqref{eqn:alphanorm} follows from Lemma~\ref{lem:unienv}. 
  \end{proof}
\medskip

\section{Kernels of tensor products}\label{sec:tensfell}

\par Let $\ma=(A_r)_{r\in G}$ and $\mb=(B_s)_{s\in H}$ be Fell bundles over the locally compact groups $G$ and $H$ respectively. Then $L^2(\ma)$ and
$L^2(\mb)$ are full right Hilbert modules over $A_e$ and $B_e$ respectively. If $\al$ is a $C^*$-norm on $\bats{\ma}{\mb}$, then $\al|_{(\bats{A_e}{B_e})}$ is a C*-norm on $\bats{A_e}{B_e}$. As shown in \cite{tens}, if $\alpha$ is the minimal (maximal) norm on $\bats{\ma}{\mb}$, then $\al|_{(\bats{A_e}{B_e})}$ is the minimal (respectively: maximal) norm on $\bats{A_e}{B_e}$.  On the other hand, $\al|_{(\bats{A_e}{B_e})}$ defines a C*-norm $\tilde{\alpha}$ on $\bats{L^2(A)}{L^2(\mb)}$, given by  $\tilde{\alpha}(\mu):=\sqrt{\alpha(\pr{\mu}{\mu})}$, $\forall \mu\in \bats{L^2(A)}{L^2(\mb)}$. Again, if $\al|_{(\bats{A_e}{B_e})}$ is the minimal (maximal) norm on $\bats{A_e}{B_e}$, then $\tilde{\alpha}$ is the minimal (maximal) norm on $\bats{L^2(\ma)}{L^2(\mb)}$.  
The completion $\bts{L^2(\ma)}{\tilde{\alpha}}{L^2(\mb)}$ of $\bats{L^2(A)}{L^2(\mb)}$ with respect to $\tilde{\alpha}$ is a full right Hilbert module over $\bts{A_e}{\alpha|_{A_e\otimes_\alpha B_e}}{B_e}$, 
where the corresponding inner
product is determined by $\pr{\xi_1\otimes\eta_1}{\xi_2\otimes\eta_2}
=\pr{\xi_1}{\xi_2}\otimes\pr{\eta_1}{\eta_2}$,  
$\forall \xi_1, \xi_2\in L^2(\ma)$, $\eta_1, \eta_2\in L^2(\mb)$.
\par In the remaining of the paper we will consider only the minimal norm $\alpha=\norm{\ }_{\textrm{min}}$ or the maximal norm $\alpha=\norm{\ }_{\textrm{max}}$ on $\bats{\ma}{\mb}$. Since in both cases we will have correspondingly that $\tilde{\alpha}=\norm{\ }_{\textrm{min}}$ and $\tilde{\alpha}=\norm{\ }_{\textrm{max}}$, we will use just the symbol $\alpha$, which will therefore be either $\norm{\ }_{\textrm{min}}$ or $\norm{\ }_{\textrm{max}}$.     
\bigskip 

We fix elements $r\in G$ and $s\in H$. Consider the map $j^{(r,s)}:C^{r}_c(\ma)\odot C^{s}_c(\mb)\to C^{(r,s)}_c(\tmin{\ma}{\mb})$ such that $\xi\odot \eta\mapsto \xi\oslash \eta$, where $\xi\oslash \eta(r',s'):=\xi(r')\otimes \eta(s')$, $\forall (r',s')\in G\times H$.  
\par For the proof of the following result we refer the reader to \cite[Lemma~4.4]{env}.

\begin{lem}\label{lem:l2}
Let \fela and \felbh be \fbs, and let $\al$ represent either the minimal or maximal C*-norm. Then there exists a unique isomorphism   
$j_2:\bts{L^2(\ma)}{\al}{L^2(\mb)}\to  L^2(\bts{\ma}{\al}{\mb}),$ 
such that $j_2(\xi\otimes\eta)=\xi\oslash\eta$, 
$\forall\xi\in C_c(\ma)\subseteq L^2(\ma)$, $\eta\in C_c(\mb)\subseteq 
L^2(\mb)$. Thus we have
$\tmin{L^2(\ma)}{L^2(\mb)}\cong L^2(\tmin{\ma}{\mb})$ and $\tmax{L^2(\ma)}{L^2(\mb)}\cong L^2(\tmax{\ma}{\mb})$.
\end{lem}

\begin{rk}\label{rk:l2}
  Note that the isomorphism $j_2^l$ induced by $j_2$ according to Proposition~\ref{prop:zhm} (see also Remark~\ref{rk:zhm}) is determined by the fact that $\forall \xi_1,\xi_2\in C_c(\ma)$ and $\forall \eta_1,\eta_2\in C_c(\mb)$ we have $j_2(\pr{\xi_1}{\xi_2}_l\odot\pr{\eta_1}{\eta_2}_l)=\pr{\xi_1\oslash\eta_1}{\xi_2\oslash\eta_2}_l$. 
\end{rk}

\begin{prop}\label{prop:l2t}
  The map $j_2^{(r,s)}:C^{r}_c(\ma)\odot C^{s}_c(\mb)\to C^{(r,s)}_c(\bts{\ma}{\alpha}{\mb})$ has a unique extension to an isomorphism of Hilbert $\bts{A_e}{\alpha}{B_e}$-modules such that the following diagram is commutative:
  \[
    \xymatrix{
      \bts{L^2(\ma)}{\alpha}{L^2(\mb)}\ar[rr]^{\quad\rho^\ma_{r}\otimes\rho^\mb_{s}}_{\cong}\ar[d]_{j_2}^\cong&& \bts{L^2_{r}(\ma)}{\alpha}{L^2_{s}(\mb)}\ar[d]^{j_2^{(r,s)}}
      \\
 L^2(\bts{\ma}{\alpha}{\mb})\ar[rr]_{\rho^{\bts{\ma}{\alpha}{\mb}}_{(r,s)}}^\cong
      &&       L^2_{(r,s)}(\bts{\ma}{\alpha}{\mb})
      }
    \]
    The unique isomorphism $(j_2^{(r,s)})^l:\bts{I_{r}(\ma)}{\alpha}{I_{s}(\mb)}\to I_{(r,s)}(\bts{\ma}{\alpha}{\mb})$ given by Proposition~\ref{prop:zhm} is determined by the fact that $\pr{\xi_1}{\xi_2}_l\odot\pr{\eta_1}{\eta_2}_l\mapsto\pr{\xi_1\oslash\eta_1}{\xi_2\oslash\eta_2}^{(r,s)}_l$, $\forall \xi\in C_c^{r}(\ma)$, $\eta\in C_c^{s}(\mb)$.   
\end{prop}
\begin{proof}
The first vertical arrow exists and is an isomorphism by Lemma~\ref{lem:l2}. The bottom horizontal arrow exists and is an isomorphism due to Proposition~\ref{prop:isomhb}. For the same reason $\rho^\ma_{r}$ and $\rho^\mb_{s}$ are isomorphisms, and hence the top horizontal arrow is an isomorphism. The commutativity of the diagram at the level of compactly supported continuous functions is clear, so  $j_2^{(r,s)}$ extends to the second vertical arrow which, in view of the previous remarks, must be an isomorphism. Finally, note that applying Proposition~\ref{prop:zhm} to the objects of the diagram above produces a commutative diagram of C*-algebras and isomorphisms (the compact operators and the maps between them induced in each case), so the last  statement follows now from Remark~\ref{rk:l2}
\end{proof}

\bigskip
\begin{thm}\label{thm:main}
  Let $\ma=(A_r)_{r\in G}$ and $\mb=(B_s)_{s\in H}$ be Fell bundles over the locally compact groups $G$ and $H$ respectively. Then
  \[\nuc{\tmin{\ma}{\mb}}=\tmin{\nuc{\ma}}{\nuc{\mb}}\qquad\textrm{ and }\qquad \nuc{\tmax{\ma}{\mb}}=\tmax{\nuc{\ma}}{\nuc{\mb}}.\] 
\end{thm}
\begin{proof}
  Let $\alpha$ be either $\norm{\ }_{\textrm{min}}$ or $\norm{\ }_{\textrm{max}}$. By the universal property of tensor products of vector spaces, we have a map $\nu:\bats{\nucc{\ma}}{\nucc{\mb}}\to \nuc{\bts{\ma}{\alpha}{\mb}}$, such that $\nu(k_1\odot k_2)\big((r_1,s_1),(r_2,s_2)\big)=k_1(r_1,r_2)\otimes k_2(s_1,s_2)$, $\forall (r_1,s_1), (r_2,s_2)\in G\times H$. When restricted to each $I_c^r(\ma)\odot I_c^s(\mb)$, $\nu$ agrees with the restriction to this set of the isomorphism $(j_2^{(r,s)})^l:\bts{I_{r}(\ma)}{\alpha}{I_{s}(\mb)}\to I_{(r,s)}(\bts{\ma}{\alpha}{\mb})$, for if $\xi_1,\xi_2\in C_c^{r}(\ma)$ and $\eta_1,\eta_2\in C_c^{s}(\mb)$, then
  \begin{gather*}
    \nu\big(\pr{\xi_1}{\xi_2}_l^{r}\odot \pr{\eta_1}{\eta_2}_l^{s}\big)\big((r_1,s_1),(r_2,s_2)\big)
    =\pr{\xi_1}{\xi_2}_l^{r}(r_1,r_2)\otimes \pr{\eta_1}{\eta_2}_l^{s}(s_1,s_2)\\
    =\xi_1(r_1)\xi_2(r_2)^*\otimes \eta_1(s_1)\eta_2(s_2)^*
    =\big(\xi_1(r_1)\otimes \eta_1(s_1)\big)\big(\xi_2(r_2)^*\otimes \eta_2(s_2)^*\big)\\
    =\big(\xi_1\oslash\eta_1(r_1,s_1)\big)\big(\xi_2\oslash\eta_2(r_1,s_1)\big)^*
    =\pr{\xi_1\oslash\eta_1}{\xi_2\oslash\eta_2}_l^{(r,s)}\big((r_1,s_1),(r_2,s_2)\big)\\
    =(j_2^{(r,s)})^l\big(\pr{\xi_1}{\xi_2}_l^{r}\odot \pr{\eta_1}{\eta_2}_l^{s}\big)\big((r_1,s_1),(r_2,s_2)\big). 
    \end{gather*}

    \par Now let $k\in\bats{\nucc{\ma}}{\nucc{\mb}}\subseteq \bts{\nuc{\ma}}{\alpha}{\nuc{\mb}}$. If $k\in I_c^{r}(\ma)\odot I_c^{s}(\mb)$, according with the computations above we have $\norm{\nu(k)}=\norm{(j_2^{(r,s)})^l(k)}=\norm{k}=\alpha(k)$, because $(j_2^{(r,s)})^l$ is isometric. Now if $k$ is arbitrary, using this and that $(j_2^{(r,s)})^l:\bts{I_{r}(\ma)}{\alpha}{I_{s}(\mb)}\to I_{(r,s)}(\bts{\ma}{\alpha}{\mb})$ is an isomorphism, plus Proposition~\ref{prop:alphanorm} and the formulas \eqref{eqn:norm} and \eqref{eqn:alphanorm} in Proposition~\ref{prop:norm} and Proposition~\ref{prop:alphanorm} respectively, we have: 
  \begin{gather*}
    \norm{\nu(k)}
    \stackrel{\eqref{eqn:norm}}{=}\sup_{(r,s)\in G\times H}\{\norm{\nu(k)*h'}:h'\in I_{(r,s)}(\bts{\ma}{\alpha}{\mb})\textrm{ and }\norm{h'}\leq 1\}\\
    =\sup_{(r,s)\in G\times H}\{\norm{\nu(k)*(j_2^{(r,s)})^l(h)}:h\in \bts{I_r(\ma)}{\alpha}{I_s(\mb)}\textrm{ and }\norm{h}\leq 1\}\\
    =\sup_{(r,s)\in G\times H}\{\norm{\nu(k)*\nu(h)}:h\in \bts{I_r(\ma)}{\alpha}{I_s(\mb)}\textrm{ and }\norm{h}\leq 1\}\\
        =\sup_{(r,s)\in G\times H}\{\norm{\nu(k*h)}:h\in \bts{I_r(\ma)}{\alpha}{I_s(\mb)}\textrm{ and }\norm{h}\leq 1\}\\    
        \stackrel{\ref{prop:alphanorm}}{=}\sup_{(r,s)\in G\times H}\{\norm{(j_2^{(r,s)})^l(k*h)}:h\in \bts{I_r(\ma)}{\alpha}{I_s(\mb)}\textrm{ and }\norm{h}\leq 1\}\\    
        =\sup_{(r,s)\in G\times H}\{\norm{k*h}:h\in \bts{I_r(\ma)}{\alpha}{I_s(\mb)}\textrm{ and }\norm{h}\leq 1\} \\
        \stackrel{\eqref{eqn:alphanorm}}{=}\norm{k}.
      \end{gather*}
      Then $\nu$ is an isometric map between the dense subspace $\nu:\bats{\nucc{\ma}}{\nucc{\mb}}$ of $\bts{\nuc{\ma}}{\alpha}{\mb}$ and its image, so it extends by continuity to an isomorphism.  
\end{proof}

\end{document}